\documentclass[oneside, 11pt]{amsart}
\usepackage[utf8]{inputenc}
\usepackage{amsfonts}
\usepackage{amsmath}
\usepackage{amssymb}
\usepackage{amsthm}
\usepackage{graphicx,todonotes}
\usepackage{xcolor}
\usepackage[colorlinks=false]{hyperref}
\newcommand{\qq}{\mathfrak{q}}

\newtheorem{thm}{Theorem}[section]

\newtheorem{lem}[thm]{Lemma}

\theoremstyle{definition}
\newtheorem{defn}[thm]{Definition}
\newtheorem{rem}[thm]{Remark}

\newtheorem{ques}[thm]{Question}

  \newcommand{\param}{{\mathchoice{\mkern1mu\mbox{\raise2.2pt\hbox{$
  \centerdot$}}
  \mkern1mu}{\mkern1mu\mbox{\raise2.2pt\hbox{$\centerdot$}}\mkern1mu}{
  \mkern1.5mu\centerdot\mkern1.5mu}{\mkern1.5mu\centerdot\mkern1.5mu}}}

\usepackage{ifthen}

\newcommand{\showcomments}{yes}

\newsavebox{\commentbox}
{\ifthenelse{\equal{\showcomments}{yes}}%
{\footnotemark
    \begin{lrbox}{\commentbox}
    \begin{minipage}[t]{1.25in}\raggedright\sffamily\upshape\tiny
    \footnotemark[\arabic{footnote}]}%
{\begin{lrbox}{\commentbox}}}%
{\ifthenelse{\equal{\showcomments}{yes}}%
{\end{minipage}\end{lrbox}\marginpar{\usebox{\commentbox}}}%
{\end{lrbox}}}

\begin{document}
\title[]{Quasi-redirecting boundaries of groups with linear divergence and  3-manifold groups}

\date{\today}

\author{Hoang Thanh Nguyen}
\address{Hoang Thanh Nguyen, Department of Mathematics, FPT University, Hoa Hai ward, Ngu Hanh Son district, Da Nang, Vietnam}
\email{nthoang.math@gmail.com}

\keywords{linear divergence, quasi-redirecting boundary, 3-manifold groups}

\subjclass[2010]{20F65, 20F67}

\begin{abstract}
The quasi-redirecting (QR) boundary, introduced by Qing and Rafi,
generalizes the Gromov boundary for studying the large-scale geometry of
finitely generated groups. Although it is not known to exist for all such groups, its existence has been established for several important classes. We prove that if a finitely generated group $G$ has linear divergence, then its QR-boundary is well-defined and consists of a single point. In addition, we show that all finitely generated 3-manifold groups admit well-defined QR-boundaries.
\end{abstract}

\maketitle

\section{Introduction}
The quasi-redirecting (QR) boundary is a close generalization of the Gromov boundary to all finitely generated groups \cite{QR24}. 
One of the advantages of the QR-boundary is that it is a new quasi-isometry invariant boundary that is often compact, containing sublinearly Morse boundaries \cite{QRT22}, \cite{QRT24} as topological subspaces, capturing a richer spectrum of hyperbolic-like behaviors, making it a promising new tool in geometric group theory.

 The QR boundary is defined as follows:
 \begin{defn}
    Let $\alpha, \beta \colon [0, \infty) \to X$ be two quasi-geodesic rays in a metric space $X$. % chktex 9
     We say $\alpha$ can be \textit{quasi-redirected} to $\beta$ (and write $\alpha \preceq \beta$) if there exists a pair of constants $(q, Q)$ such that for every $r >0$, there exists a $(q, Q)$--quasi-geodesic ray $\gamma$ that is identical to $\alpha$ inside the ball $B(\alpha(0), r)$ and eventually $\gamma$ becomes identical to $\beta$. We say $\alpha \sim \beta$ if $\alpha \preceq \beta$ and $\beta \preceq \alpha$. 
     
     The resulting set of equivalence classes forms a poset, denoted by $P(X)$ (in this paper, we will call it the \textit{QR-poset}). This poset $P(X)$, when equipped with a ``cone-like topology'' (see \cite[Section 5]{QR24}), is called the \textit{quasi-redirecting boundary} (QR boundary) of $X$ and denoted by $\partial_{*} X$.
\end{defn}

Qing and Rafi \cite{QR24} established key properties of the QR-boundary, with further developments in \cite{GQV24}.
While QR boundaries are shown to be well-defined for several classes of groups of interest, including relatively hyperbolic groups, Croke-Kleiner admissible groups, non-geometric 3-manifold groups \cite{QR24}, \cite{NQ25}, its existence for all finitely generated groups remains an open question.

\begin{ques}\cite[Question D]{QR24}
  Let $X$ be a Cayley graph of a finitely generated group. Is $\partial_{*} X$ always defined? Is $\partial_{*} X$ always compact?  
\end{ques}

In \cite[Section 4]{QR24}, the authors show that the QR-boundary of a direct product of two infinite finitely generated groups consists exactly of one point. They also mention the work in \cite{McM} where the author shows the same holds for Baumslag-Solitar group. These groups share a common
 feature: linear divergence, a quasi-isometry invariant that measures the minimal
 path length outside a ball connecting two points on its boundary, as a function
 of the ball’s radius  \cite{Gro96}, \cite{Ger94a}, \cite{DMS10} (see the precise definition in Section~\ref{sec:divergence}).

This observation naturally raises the following question:

\begin{ques}
\label{ques:main}
    For a finitely generated group $G$, does linear divergence imply that $\partial_* G$ is a single point?
\end{ques}

Our result provides an affirmative answer to Question~\ref{ques:main}. 

\begin{thm}
    \label{thm:main1} Let $G$ be a finitely generated group. If $G$ has linear divergence, then the QR-boundary $\partial_{*} G$ consists of exactly one point.
\end{thm}

This result confirms the existence and triviality of the QR-boundary for groups
 with linear divergence, adding to known examples such as:

\begin{itemize}
    \item Lattices in semi-simple Lie groups of $\mathbb Q$--rank $1$ and $\mathbb R$--rank $\ge 2$, uniform lattices in higher rank semi-simple Lie groups \cite{DMS10}.
    \item Thompson groups $F$, $T$, and $V$ \cite{GS19} and higher Thompson groups \cite{Kod24}.
    \item Non-virtually cyclic groups that satisfy a law \cite{DS05}, \cite{DMS10}.
    \item One-ended solvable groups \cite{DS05}, \cite{DMS10}.
    \item Wreath products, permutational wreath products of groups, Houghton groups $\mathcal{H}_m$ with $m \ge 2$, Baumslag-Solitar groups \cite{I23}.
\end{itemize}

Linear divergence is equivalent to \textit{wide} (i.e., not having cut-points in the asymptotic cones) \cite{DMS10}, and wide groups have empty Morse boundary \cite{DMS10}.
\cite[Question 4.4]{QR24} asks if $G$ does not have an Morse element, is $P(G)$ a single
point. In \cite{GQV24}, the authors answer \cite[Question 4.4]{QR24} in the affirmative when $G$ acts geometrically on a finite-dimensional CAT(0) cube complex. Our result Theorem~\ref{thm:main1} gives the affirmative  for \cite[Question 4.4]{QR24} for the class of wide groups.

\cite[Theorem A, Theorem B] {NQ25} show that the QR poset of graph
 manifold groups, and more generally of Croke-Kleiner admissible groups \cite{CK02},
 has QR-poset of height $2$. This is connected to the fact that these groups have quadratic
 divergence. Our result shows that groups with linear divergence have QR
 poset of height $1$. This naturally raises the question of whether there is a
 systematic relationship between divergence and QR-boundary structure.
 \begin{ques}
     If a group has divergence that is a polynomial of degree $d$, is
 it true that its QR poset has height $d$?
 \end{ques}

As an application of Theorem~\ref{thm:main1}, we establish a comprehensive result for finitely generated 3-manifold groups:
\begin{thm}
    \label{thm:main2}
 All finitely generated 3-manifold groups have well-defined QR-boundaries.  
\end{thm}

This result addresses cases left unresolved in \cite{NQ25}. While \cite[Theorem A]{NQ25} showed that QR-boundaries are well-defined for fundamental groups of non-geometric 3-manifolds, the existence of QR-boundaries for geometric 3-manifolds--particularly those modeled on the Sol and Nil geometries and the broader scenario of 3-manifolds with higher genus boundaries was not completely settled.

These cases were excluded precisely because it was unknown whether their fundamental groups Sol and Nil 3-manifolds satisfy the necessary QR-assumptions.  Since Sol and Nil 3-manifold groups are known to have linear divergence \cite{Ger94}, it follows from Theorem~\ref{thm:main1} that their fundamental groups have well-defined QR-boundary. We use this observation as a first step to conclude that all finitely generated 3-manifold groups admit a well-defined QR-boundary. Theorem~\ref{thm:main2} strengthens the role of QR boundaries as a tool for studying the coarse geometry of finitely generated 3-manifold groups, one of the central topics in geometric group theory.

\subsection*{Overview} This paper is organized as follows. Section~\ref{sec:preliminary} reviews preliminary concepts, including the QR-boundary construction and divergence. Section~\ref{sec:preliminary} proves Theorem~\ref{thm:main1}, demonstrating that groups with linear divergence have a single-point QR-boundary. In Section~\ref{sec:3-manifold}, we give a proof of Theorem~\ref{thm:main2}.

\subsection*{Acknowledgments} 
The author thanks Yulan Qing and Minh Nhat Doan for the helpful conversations.

\section{Preliminary}
\label{sec:preliminary}

\subsection{Coarse geometry}
In this section, we recall the construction of quasi-redirecting boundary as presented in~\cite{QR24}.
Let $X$ and $Y$ be metric spaces and $f$ be a map from $X$ to $Y$. Let ${\qq}=(q, Q) \in[1,\infty) \times [0,\infty)$ be a pair of constants. 

\begin{defn}
\begin{enumerate}
	
	\item We say that $f$ is a \emph{$(q, Q)$--quasi-isometric embedding} if  for all $x, y\in X$,
	      \[
		        \frac{1}{q} d(x, x') - Q \le d(f(x), f(x')) \le q d(x,x') + Q.
	      \]

	\item We say that $f$ is a \emph{$(q,Q)$--quasi-isometry} if it is a $(q, Q)$--quasi-isometric embedding such that $Y = N_{Q}(f(X))$.	
\end{enumerate}
\end{defn}

\begin{defn}\label{Def:Quadi-Geodesic} 
A \emph{quasi-geodesic} in a metric space $X$ is a quasi-isometric embedding  
$\alpha: I \to X$ where $I \subset \mathbb{R}$ is a  (possibly infinite) interval.  That is    $\alpha :  I \to X$ is a $(q, Q)$--quasi-geodesic if for all $s, t \in I$, we have 
\[
\frac{|t-s|}{q} - Q  \leq d_X \big(\alpha(s), \alpha(t)\big)  \leq q |s-t| + Q 
\]
\end{defn}

\begin{rem}
We can always assume $\alpha$ is $(2q + 2Q)$–Lipschitz, and hence,
$\alpha$ is continuous. By \cite[Lemma 2.3]{QR24} the Lipschitz assumption can be made
without loss of generality. 
\end{rem}

\subsection*{Notation:} 
Let $o$ be a fixed base-point in  $X$. We use ${\qq}=(q, Q) \in[1, \infty) \times[0, \infty)$ to indicate a pair of constants.
 For instance, one can say $\Phi \colon X \to Y$ is a ${\qq}$--quasi-isometry and $\alpha$ is a ${\qq}$--quasi-geodesic ray or segment. 
\begin{itemize}
    \item  By a \textit{${\qq}$--ray} we mean a ${\qq}$--quasi-geodesic ray $\alpha:[0, \infty) \to X$ such that $\alpha(0)=  o$.
    \item 
If points $x, y \in X$ on the image of $\alpha$ are given, we denote the sub-segment of $\alpha$ connecting $x$ to $y$ by $[x, y]_\alpha$. 
\item For $r>0$, let $B_r^{\circ} \subset X$ be the open ball of radius $r$ centered at $o$, let $B_r$ be the closed ball centered at $o$ and let $B_r^c=X-B_r^{\circ}$. For a $\qq$--ray $\alpha$ and $r>0$, we let $t_r \geq 0$ denote the first time when $\alpha$ first intersects $B_r^c$.

Lastly, if $p$ is a point on a ${\qq}$--ray $\alpha$,  we use $\alpha_{[p, \infty)}$ to denote the tail of $\alpha$ starting from the point $p$.
\end{itemize}

\subsection{QR-Assumptions}
In this section, we briefly review the notion \textit{QR-poset} and \textit{QR-redirecting boundary} from \cite{QR24}.
\begin{defn}\label{Def:Redirection}
Let $X$ be a geodesic metric space.
Let $\alpha, \beta$ and $\gamma$ be quasi-geodesic rays in $X$. We say
\begin{enumerate}
    \item $\gamma$ \emph{eventually coincides with 
$\beta$} if there are times $t_\beta, t_\gamma >0$ such that, 
for $t\geq t_\gamma$, we have 
$
\gamma(t) =\beta(t+t_\beta)
$.

\item For $r>0$, we say $\gamma$ \emph{quasi-redirects $\alpha$ to $\beta$ at radius $r$} if
$
\gamma|_r = \alpha|_r$  and $\beta$ 
eventually coincides with $\gamma$.
 If $\gamma$ is a ${\qq}$--ray, we say \emph{$\alpha$ can be ${\qq}$--quasi-redirected to $\beta$
at radius $r$} or \textit{$\alpha$ can be ${\qq}$--quasi-redirected to $\beta$
by $\gamma$ at radius $r$}. We refer to $t_\gamma$ as the \emph{landing time}. 

\item We say $\alpha$ is \textit{quasi-redirected} to $\beta$, denoted by $\alpha \preceq \beta$, if there is ${\qq} \in [1, \infty) \times[0,\infty)$ such that for every $r>0$, $\alpha$ can be ${\qq}$--quasi-redirected to $\beta$ at radius $r$. 
\end{enumerate}
\end{defn}

\begin{defn}
Define $\alpha \simeq \beta$ if and only if $\alpha  \preceq \beta$ and $\beta  \preceq \alpha$. Then  
$\simeq$ is an equivalence relation on the space of all quasi-geodesic rays in $X$. 

Let $P(X)$
denote the set of all equivalence classes of quasi-geodesic rays under $\simeq$. 
For a quasi-geodesic ray $\alpha$, let $[\alpha] \in P(X)$ denote the equivalence class containing 
$\alpha$. We extend $\preceq$ to $P(X)$ by defining $[\alpha] \preceq [\beta]$ if 
$\alpha \preceq \beta$. Note that this does not depend on the chosen representative
in the given class. The relation $\preceq$ is a partial order on elements of $P(X)$. We call $P(X)$ the \textit{QR-poset} of $X$.
\end{defn}

\underline{QR-Assumption 0:} (No dead ends)

The metric space $X$ is proper and geodesic. 
Furthermore, there exists a pair of constants ${\qq}_0$ such that every point $x \in X$ lies on an infinite ${\qq}_0$--quasi-geodesic ray.  

\underline{QR-Assumption 1:} (Quasi-geodesic representative)

For ${\qq}_0$ as in QR-Assumption 0, every equivalence class of quasi-geodesics $\mathbf{a} \in P(X)$ contains a ${\qq}_0$--ray. 
We fix such a ${\qq}_0$--ray, denote it by $\underline{a} \in \mathbf{a}$, and call it a \emph{central element} of $\mathbf{a}$.

\underline{QR-Assumption 2:} (Uniform redirecting function) 

For every $\mathbf{a} \in P(X)$, there is a function 
\[
f_\mathbf{a} : \, [1, \infty) \times [0, \infty) \to [1, \infty) \times [0, \infty), % chktex 9 chktex 44
\] 
called the redirecting function of
the class $\mathbf{a}$, such that if $\mathbf{b} \prec \mathbf{a}$ then any ${\qq}$--ray $\beta \in \mathbf{b}$ can be
$f_\mathbf{a}({\qq})$--quasi-redirected to $\underline{a}$. 

\underline{Quasi-redirecting boundary (QR-boundary):}

Once a proper geodesic metric space $X$ satisfies all three QR-Assumptions, there is a ``cone-line'' topology on the poset $P(X)$ described on \cite{QR24}. This poset $P(X)$, when equipped with this topology, is called the \textit{quasi-redirecting boundary} (QR boundary) of $X$ and denoted by $\partial_{*} X$. Since we don't use this topology on $P(X)$ in an essential way in this paper, we refer the reader to \cite{QR24} for the detailed discussion.

A remarkable fact about QR-boundary is the following result.

\begin{thm}[{\cite[Theorem B, Theorem C]{QR24}}]\label{thm:collectfacts}
Let $X, Y$ be proper geodesic metric spaces satisfying all three QR-Assumptions.
\begin{enumerate}
\item A quasi-isometry $f \colon X \to Y$ induces a
homeomorphism between $\partial_{*} X$ and $\partial_{*} Y$.
\item Sublinearly Morse boundaries are topological subspaces of $\partial_{*} X$.
\end{enumerate}
\end{thm}

\subsection{Divergence of groups}
\label{sec:divergence}
In this section, we briefly review  the definition of divergence from \cite{DMS10}.

\begin{defn}
\label{def:equivalentfunction}
Let $\mathcal{F}$ be the collection of all functions from positive reals to positive reals. Let $f$ and $g$ be arbitrary elements of $\mathcal{F}$. The function $f$ is \emph{dominated} by a function $g$, denoted by
\emph{$f\preceq g$}, if there are positive constants $A$, $B$, $C$, $D$ and $E$ such that
\[
  f(x) \leq A\,g(Bx+C)+Dx+E \quad \text{for all $x$.}
\]
Two functions $f$ and $g$ are \emph{equivalent},
denoted by $f\sim g$, if $f\preceq g$ and $g\preceq f$.
\end{defn}

\begin{rem}
The relation $\sim$ is an equivalence relation on the set $\mathcal{F}$. Let $f$ and $g$ be two polynomial functions with degree at least $1$ in $\mathcal{F}$, then it is not hard to show that they are equivalent if and only if they have the same degree. Moreover, all exponential functions of the form $a^{bx+c}$, where $a>1$, $b>0$ are equivalent.
\end{rem}

Since we mainly work on Cayley graphs of finitely generated groups, we assume in this section that our metric spaces are geodesic, proper and periodic spaces.

For such a space $X$, given three points $a, b,c \in X$ and parameters $\delta \in (0,1)$ and $ \gamma \ge 0$ we define \emph{divergence} $\operatorname{div}_{\gamma} (a,b,c ; \delta)$ to be the infimum of lengths of paths that connect $a$ to $b$ outside $B^{o}(c, \delta r - \gamma)$, the open ball around $c$ of radius $\delta r - \gamma$, if this exists. We define it to be infinite otherwise. Here $r = \operatorname{min} \{ d(c, a), d(c, b) \}$. We then define
$$
\operatorname{Div}(n, \delta) = \operatorname{sup}_{a,b,c \in X, d(a,b) \le n } \operatorname{div}(a, b, c; \delta)
$$

Since a space has more than one end then its divergence is infinite, we thus restrict to one-ended spaces.
 Furthermore, we can fix a third point $c = x_0$ in the definition and assume that $a,b$ are in the sphere $S_r : = S(x_0, r)$ and $\operatorname{Div}_{\gamma}(n, \delta)$ can be modified to:
$$
\operatorname{Div}_{\gamma}(n, \delta) = \operatorname{sup}_{a, b \in S(x_0, r)} \operatorname{div}_{\gamma} (a, b, x_0; \delta)
$$

It is shown by \cite{DMS10} that $\operatorname{Div}_{\gamma} (n ,  \delta)$ is independent of $\gamma$ and $\delta$ up to $\sim$ for any $\delta \le 1/2$ and $\gamma \ge 2$ and is invariant under quasi-isometry up to $\sim$. Thus in this paper, we think of $\operatorname{Div}(X)$ as a function of $n$, defining it to be equal to $\operatorname{Div}_{2}(n, 1/2)$. We say that the divergence is linear if $\operatorname{Div}_{2}(n, 1/2) \sim n$, quadratic if $\operatorname{Div}_{2}(n, 1/2) \sim n^2$, and so on.

\section{Quasi-redirecting boundaries of groups with linear divergence}
\label{sec:lineardivergence}
In this section, we are going to prove Theorem~\ref{thm:main1}. To prove Theorem~\ref{thm:main1}, we show that the cayley graph $X$ of $G$ (with respect to a finite generating set), any two quasi-geodesic rays $\alpha$ and $\beta$  are equivalent under the QR-relation. We achieve this by constructing quasi-geodesic rays that direct $\alpha$ to $\beta$ by using the linear divergence property.

We need the following lemmas.

\begin{lem}[{\cite[Lemma 2.6]{QR24}}]
\label{concate}
Let $X$ be a metric space that satisfies QR-Assumption 0. 
 (Nearest-point projection surgery)
Consider a point $x \in X$ and a $(q, Q)$--quasi-geodesic segment $\beta$ 
connecting a point $z \in X$ to a point $w \in X$. Let $y$ be a closest point in $\beta$
to $x$. Then 
\[
\gamma = [x, y] \cup [y, z]_\beta
\] 
is a $(3q, Q)$--quasi-geodesic.
\end{lem}

\begin{lem} \cite[Lemma 2.9]{NQ25}
\label{lem:infinitepoints}
Let $\alpha,\beta$ be quasi-geodesic rays. Suppose there exists constants  $\qq$ and  a sequence of points $\{x_n\}$ on $\alpha$ such that $\operatorname{norm}{x_n} \to \infty$ and the following holds. For every 
  $n$, there exists a $\qq$-ray $\gamma_n$ such that
  $\gamma_n$ eventually coincides with $\beta$, and $\gamma_n$ and $\alpha$ are identical on the subsegment $[o, x_{n}]_{\alpha}$. Then  $\alpha$ can be $\qq$-quasi-redirected to $\beta$.
\end{lem}

The following lemma follows from the proof of \cite[Lemma 3.5]{QR24}.

\begin{lem} \cite[Lemma 3.5]{QR24}
\label{lem:geodesicrepresentativeunder}
Let $X$ be a proper, geodesic, metric space and let $\alpha$ be a $\qq$--ray. Then there
exists a geodesic ray $\bar{\alpha}$ such that $\bar{\alpha}$ is $(3q, Q)$--quasi-redirected to $\alpha$.
\end{lem}

\begin{lem} \cite[Lemma 3.3]{Tra19}]
\label{ll2}
For each $C>1$ and $\rho \in (0,1]$ there is a constant $L=L(C,\rho)\geq 1$ such that the following holds. Let $r$ be an arbitrary positive number and $\gamma$ a  path with the length  $\ell(\gamma) <Cr$ and $d(x,y) >r$. Then there is an $(L,0)$--quasi-geodesic $\alpha$ connecting two points $x$, $y$ such that the image of $\alpha$ lies in the $\rho r$--neighborhood of $\gamma$ and $\ell(\alpha) < \ell(\gamma)$.
\end{lem}

\begin{lem}[Annulus Surgery]
\label{lem:keylemma}
Let $X$ be a proper geodesic metric space. Given $\delta > \epsilon >0$ and a constant $C >0$. Given two pairs of constants $(q_1, Q1), (q_2, Q_2) \in [1, \infty) \times [0, \infty)$. Then there exists a constant $M = M(\delta, \epsilon, q_1, Q_1, q_2, Q_2, C)$ such that the following holds. 

Let $\alpha$ be a $(q_1, Q_1)$--quasi-geodesic based at $o$ with $\alpha_{+}$ in the sphere  $S_{\epsilon r} : =  S(o, \epsilon r)$ and lies entirely in the ball $B_{\epsilon r}$ with $r > 1$. Let $\zeta$ be a geodesic ray based at $o$ and passes through $\alpha_{+}$. Let $p$ denote the intersection point of $\zeta$ with the sphere $S_{\delta r}$. Let $\gamma$ be a $(q_2, Q_2)$--quasi-geodesic based at $p$ such that $\beta$ lies entirely outside the open ball $B^{o}_{\delta r}$ and $\ell(\beta) \le Cr$. Then the concatenation $\sigma : = \alpha \cup [\alpha_{+}, p] \cup \beta$ is a $(M, M)$--quasi-geodesic.
\end{lem}

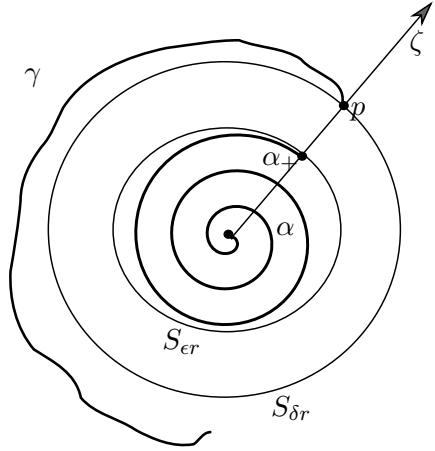
\begin{figure}[htb]
\centering 
%\def\svgwidth{\columnwidth}
%\input{star.pdf_tex}
 %\resizebox{0.7\textwidth}{!}{\input{Fig2.pdf_tex}}
 \def\svgwidth{0.5\textwidth}
 %% Creator: Inkscape inkscape 0.92.5, www.inkscape.org
%% Accompanies image file 'QR1.pdf' (pdf, eps, ps)
\begingroup%
  \makeatletter%
  \providecommand\color[2][]{%
    \errmessage{(Inkscape) Color is used for the text in Inkscape, but the package 'color.sty' is not loaded}%
    \renewcommand\color[2][]{}%
  }%
  \providecommand\transparent[1]{%
    \errmessage{(Inkscape) Transparency is used (non-zero) for the text in Inkscape, but the package 'transparent.sty' is not loaded}%
    \renewcommand\transparent[1]{}%
  }%
  \providecommand\rotatebox[2]{#2}%
  \newcommand*\fsize{\dimexpr\f@size pt\relax}%
  \newcommand*\lineheight[1]{\fontsize{\fsize}{#1\fsize}\selectfont}%
  \ifx\svgwidth\undefined%
    \setlength{\unitlength}{484.33504334bp}%
    \ifx\svgscale\undefined%
      \relax%
    \else%
      \setlength{\unitlength}{\unitlength * \real{\svgscale}}%
    \fi%
  \else%
    \setlength{\unitlength}{\svgwidth}%
  \fi%
  \global\let\svgwidth\undefined%
  \global\let\svgscale\undefined%
  \makeatother%
  \begin{picture}(1,0.93551711)%
    \lineheight{1}%
    \setlength\tabcolsep{0pt}%
    \put(0.03223306,0.76828072){\color[rgb]{0,0,0}\makebox(0,0)[lt]{\lineheight{1.25}\smash{\begin{tabular}[t]{l}$\gamma$\end{tabular}}}}%
    \put(0,0){\includegraphics[width=\unitlength,page=1]{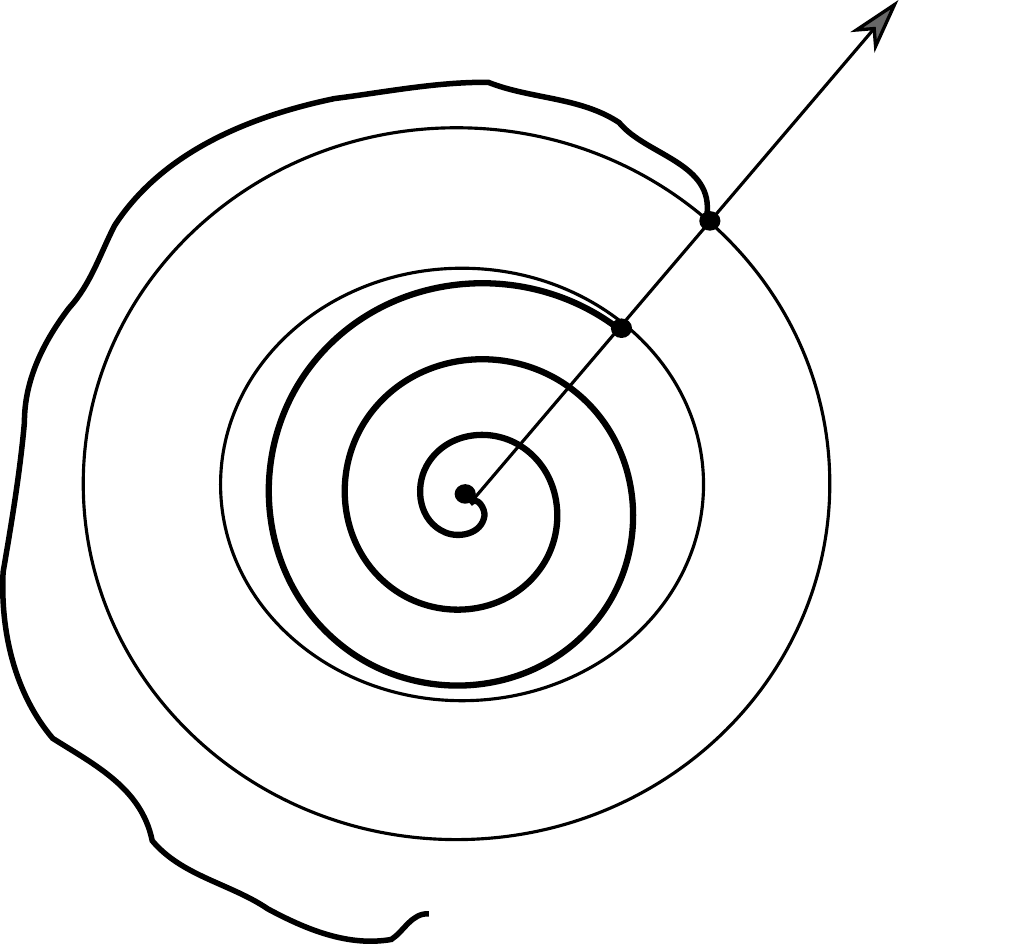}}%
    \put(0.84143827,0.83269976){\color[rgb]{0,0,0}\makebox(0,0)[lt]{\lineheight{1.25}\smash{\begin{tabular}[t]{l}$\zeta$\end{tabular}}}}%
    \put(0.56143998,0.4447692){\color[rgb]{0,0,0}\makebox(0,0)[lt]{\lineheight{1.25}\smash{\begin{tabular}[t]{l}$\alpha$\end{tabular}}}}%
    \put(0.55049029,0.06465986){\color[rgb]{0,0,0}\makebox(0,0)[lt]{\lineheight{1.25}\smash{\begin{tabular}[t]{l}$S_{\delta r}$\end{tabular}}}}%
    \put(0.32211184,0.20387682){\color[rgb]{0,0,0}\makebox(0,0)[lt]{\lineheight{1.25}\smash{\begin{tabular}[t]{l}$S_{\epsilon r}$\end{tabular}}}}%
    \put(0.71786363,0.69661124){\color[rgb]{0,0,0}\makebox(0,0)[lt]{\lineheight{1.25}\smash{\begin{tabular}[t]{l}$p$\end{tabular}}}}%
    \put(0.53171942,0.59180743){\color[rgb]{0,0,0}\makebox(0,0)[lt]{\lineheight{1.25}\smash{\begin{tabular}[t]{l}$\alpha_{+}$\end{tabular}}}}%
  \end{picture}%
\endgroup%

\caption{ The cacatenation $\alpha \cup [\alpha_{+}, p] \cup \gamma$ is a quasi-geodesic.
}\label{annulus}
\end{figure}

\begin{proof}
To see this, for every $x \neq y \in \sigma$, we are going to show that the ratio $$\frac{\ell([x,y]_{\sigma})} { d(x,y)}$$ is bounded above by a uniform constant.

According to Lemma~\ref{concate}, the concatenations $$\alpha \cup [\alpha_{+}, p], [\alpha_{+}, p] \cup \gamma$$ are $(3q_1, Q_1)$--quasi-geodesic and $(3q_2, Q_2)$--quasi-geodesic respectively.

We thus only need to consider the case $x \in \alpha$ and $y \in \gamma$. On a one hand, we have

\begin{align*}
    \ell([x,y]_{\sigma}) &\le \ell (\alpha) + d(\alpha_{+}, p) + \ell (\gamma) \\
    &\le q_{1} d(o, \alpha_{+}) + Q_1 + (\delta - \epsilon) r + Cr \\
    &\le (q_1 + \delta -\epsilon + C)r + Q_1 < (q_1 + \delta - \epsilon + C + Q_1) r
\end{align*}

On the other hand, since $x, y$ lie outside the annulus $B_{\delta r} \backslash B^{o}_{\epsilon r}$, we have
$$
d(x, y) \ge (\delta - \epsilon) r
$$
Thus we have
$$
\frac{\ell([x,y]_{\sigma})}{d(x,y)} \le \frac{(q_1 + Q_1 + C + \delta - \epsilon)r}{(\delta - \epsilon) r} = \frac{q_1 + Q_1 + C + \delta - \epsilon}{\delta - \epsilon }
$$
Combining with cases $x , y \in \alpha \cup [\alpha_{+}, p]$ , $x , y \in  [\alpha_{+}, p] \cup \gamma]$ (which are $(3q_1, Q_1)$ and $(3q_2, Q_2)$--quasi-geodesics respectively), there is a constant $M$ depending only on constants $q_1, Q_1, q_2, Q_2, \delta, \epsilon, C$ so that $\sigma$ is a $(M, M)$--quasi-geodesic.
\end{proof}

The following lemma is a slight modification of \cite[Lemma 4.3]{QRT22}. We include it here for completeness.

\begin{lem} 
    \label{redirect11} (Quasi-geodesic ray to geodesic ray surgery) Let $\epsilon \in (0,1)$. Let $\beta$ be a geodesic ray and 
$\gamma$ be a $(q, Q)$--ray. Suppose that there exists an increasing sequence $\{r_n\}$ such that  for every $n$,  $d(\beta (r_n), \gamma)\leq \epsilon r_n$. Then $\gamma$ can be $(9q,Q)$--quasi-redirected to the geodesic ray $\beta$.
\end{lem}

\begin{proof}
    Let $q_n$ be a point in $\gamma$  that is closest to $\beta(r_n)$ and let $R_n >0$ be such that the ball of radius
$R_n$ centered at $o$ contains $[o, q_n]_{\gamma}$. Now let $q'_n$ be the point in $[o, q_n]_{\gamma}$ closest to $\beta(R_n)$. Then
\begin{align*}
    d(o, q'_n) &\ge d(o, \beta(R_n)) - d(\beta(R_n), q'_n) = R_n - d(\beta(R_n), q_n) \\
    &\ge R_n - ( d(\beta(R_n), \beta(r_n)) + d(\beta(r_n), q_n)) \\
    &\ge R_n - (R_n -r_n) - \epsilon r_n = (1 - \epsilon) r_n 
\end{align*}
By Lemma~\ref{concate}, the concatenation $\zeta: = [o, q'_n]_{\gamma} \cup [q'_n, \beta(R_n)]$ is a $(3q, Q)$--quasi-geodesic. Furthermore, $d(o, q'_n) \le R$, it follows that the projection of any point on the geodesic $\beta ([R_n, \infty))$ to $\zeta$ is the point $\beta(R_n)$. By Lemma~\ref{concate} again, the concatenation $\xi : = \zeta \cup \beta ([R_n, \infty)$  is  a $(9q, Q)$--quasi-geodesic. Since $d(o, q'_n) \ge (1 - \epsilon) r_n$, it follows that $\xi$ is identical with $\gamma$ in the open ball $B^{o}_{(1 -\epsilon) r_n}$. In other words, $\gamma$ can be $(9q, Q)$--quasi-redirected to the geodesic ray $\beta$ at the radius $r'_n : = (1- \epsilon) r_n$. As $r'_n \to \infty$, it follows from Lemma~\ref{lem:infinitepoints} that $\gamma$ can be $(9q, Q)$--quasi-redirected to $\beta$.  The lemma is proved.
\end{proof}

We are now ready for the proof of Theorem~\ref{thm:main1}.

\begin{proof}[{Proof of Theorem~\ref{thm:main1}}]
Fix a finite generating set for $G$, and let $X$ the Cayley graph of $G$ with respect to this generating set.
We aim to prove that the poset $P(X)$ consists of exactly one point and satisfies all three QR-Assumptions.

Consider two $\qq$-rays $\alpha$ and $\beta$ in $X$, both based at a vertex $o$. By Lemma~\ref{lem:geodesicrepresentativeunder}, there exist geodesic rays $\bar{\alpha}$ and $\bar{\beta}$ in $X$ such that $\bar{\alpha} \preceq \alpha$ and $\bar{\beta} \preceq \beta$. We will show that $\alpha \preceq \bar{\beta}$, which implies $\alpha \preceq \beta$. By a symmetric argument, we can also establish $\beta \preceq \alpha$, thus proving $\alpha \sim \beta$ (and hence the poset $P(X)$ consists of exactly one point).

\begin{figure}[htb]
\centering 
%\def\svgwidth{\columnwidth}
%\input{star.pdf_tex}
 %\resizebox{0.7\textwidth}{!}{\input{Fig2.pdf_tex}}
 \def\svgwidth{0.9\textwidth}
 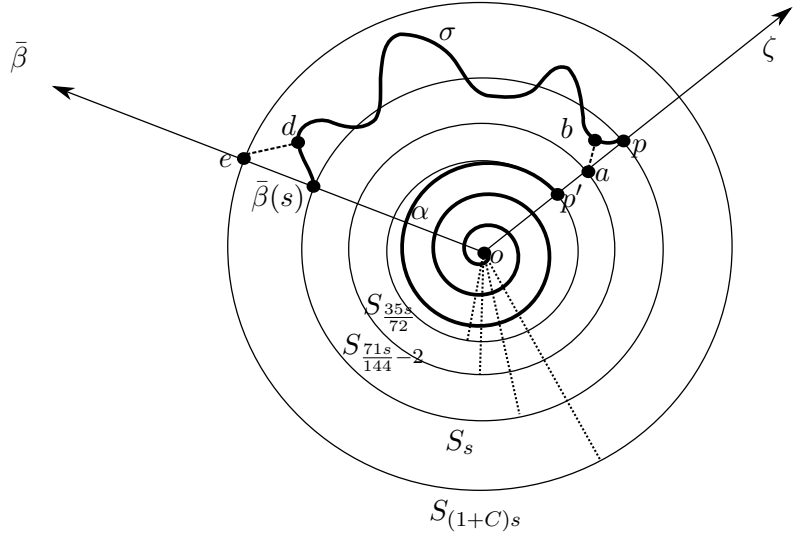
\caption{
The figure demonstrates the concatenation $\gamma : = [o, p']_{\alpha} \cup p', a] \cup [ a, b] \cup [b, d]_{\sigma} \cup [d, e] \cup \bar{\beta}|_{[e, \infty)}$ is the desired quasi-geodesic which quasi-redirects $\alpha$ to $\bar{\beta}$ at radius $\frac{35s}{72}$
}\label{linear}
\end{figure}

Since the divergence of $X$ is linear, it follows that there exists a constant $C >0$ such that for every $r >1$ then
$$
\operatorname{Div}_{2} (r, 1/2) \le C r
$$
That is
for every $x  \neq y $ in the sphere $S(o,r)$, there exists a path $\alpha$  connecting $x$ to $y$ lying outside the open ball $B^{o}_{\frac{r}{2} -2}$ such that $\ell(\alpha) \le C r$.

Recall that $t_r$ denote the first time the path $\alpha$ first intersects $B^{c}_{r} : = X  \backslash B^{o}_r$. Since $\bar{\beta}$ is a geodesic ray, we thus have $\bar{\beta}(t_r) = \bar{\beta}(r)$.

We consider the following cases.

\textbf{Case~1:} 
There exists a strictly increasing sequence $(r_n)_{n \in \mathbb{N}}$ such that for each $n \in \mathbb{N}$, the inequality 
$$
d(\bar{\beta}(r_n), \alpha(t_{\frac{35r_n}{72}})) \leq \frac{39r_n}{72}
$$ holds. 
It follows that 
$$d(\beta(r_n), \alpha) \le \frac{39r_n}{72} 
$$
By Lemma~\ref{redirect11}, $\alpha$ can be $(9q, Q)$--quasi-redirected to $\bar{\beta}$.

\textbf{Case~2:} Case~1 does not hold, that is, for every increasing sequence $r_1 < r_2 < \ldots$, there exists $n_0 \in \mathbb N$ such that 
$$
d(\bar{\beta}(r_{n_0}), \alpha(t_{\frac{35r_{n_0}}{72}})) > \frac{39r_{n_0}}{72}
$$ 
It follows that there exists an increasing subsequence $4< s_1 < s_2 < \ldots$ of $\{r_n\}$ such that
$$
d(\bar{\beta}(s), \alpha(t_{\frac{35s}{72}
})) > \frac{39s}{72} 
$$ for all $s = s_i$.

 Let $\zeta$ be a geodesic ray based at $o$ and passes through $p': = \alpha(t_{\frac{35s}{72}})$. Let $p$ denote the intersection point of $\zeta$ with $S(o, s)$. We have
$$ d(\bar{\beta}(s), p) \ge d (\bar{\beta}(s), p') - d(p', p) 
    \ge \frac{39s}{72} - (s - \frac{35s}{72})= \frac{s}{36}$$

    \textit{Claim~1:}   There exists a constant $L = L(C)$, and  an $(L,0)$--quasi-geodesic $\sigma$  connecting $p$ to $\bar{\beta}(s)$ such that $\ell(\sigma) < Cs$ and $\sigma$ lies in the annulus $B_{(1+C)s} \backslash B^{o}_{\frac{71s}{144}-2}$.

Indeed, since both points $\bar{\beta}(s)$ and $p$ lie in the sphere $S(o, s)$ and the divergence of $X$ is linear, it follows that there is a path $\tau$ connecting $p$ to $\bar{\beta}(s)$ lying outside the open ball $B^{o}(o, s/2 -2)$ and the length of $\tau$ satisfies $$\ell(\tau) \le Cs$$ 

By Lemma~\ref{ll2}, there exists a constant $L$ depending only on $C$, an $(L,0)$--quasi-geodesic $\sigma$ from $p$ to $\bar{\beta}(s)$ so that $\ell(\sigma) < \ell(\tau) < Cs$ and $\sigma \subset \mathcal{N}_{\frac{s}{144}} (\tau)$ where $\mathcal{N}_{\frac{s}{144}}(\tau)$ we mean the $\frac{s}{144}$--neighborhood of $\tau$.
Since $\tau$ lies outside the ball $B(o, s/2 -2)$, it follows that $\sigma$ lies outside the ball $B_{\frac{71s}{144} -2}$.

Let $R := \operatorname{max} \{ d(x , o) : x \in \sigma \}$.Then we have $\sigma \subset B_{R}$. Note that $R \le \ell(\sigma) + d(\bar{\beta}(s), o) \le Cs + s = (1+C)s$. Thus $\sigma$ lies inside the ball $B_{(1+C)s}$.

In the rest of the proof, we will construct the desired quasi-geodesic. 

Let $a$ be the intersection point of the geodesic ray $\zeta$ with the sphere $S_{\frac{71s}{144} -2}$. Let $b \in \sigma$ be the nearest point to $a$.  Since $\sigma$ is an $(L,0)$--quasi-geodesic, it follows that the concatenation $ \gamma_1 : = [a,b] \cup [b, \bar{\beta}(s)]_{\sigma}$ is a $(3L, 0)$--quasi-geodesic by Lemma~\ref{concate}.
 
Let $e : = \bar{\beta}((1+C)s)$, and let $d \in \gamma_1 =  [a,b] \cup [b, \bar{\beta}(s)]_{\sigma}$ be the nearest point to $e$ (we refer the reader to Figure~\ref{linear}). Again by Lemma~\ref{concate}, we have that the concatenation $\gamma_2 : = [e, d] \cup [d, a]_{\gamma_1}$ is a $(9L,0)$--quasi-geodsic.

We have
\begin{align*}
    \ell(\gamma_2) & \le d(e,d) + d(a,b) + \ell(\sigma) \\
    &\le d(e, \bar{\beta}(s)) + d(a, p) + Cs  \\
    &\le Cs + 2 + \frac{74s}{144} + Cs \\
    &\le (\frac{75}{144} + 1 + 2C)s
\end{align*}

Applying Lemma~\ref{lem:keylemma}, we have that the concatenation 
$$
\xi: = [o, p']_{\alpha} \cup [p', a] \cup \gamma_2
$$ is a $(M, M)$--quasi-geodesic where $M$ is the constant given by Lemma~\ref{lem:keylemma}.

Now, every point $u \in \bar{\beta}|_{(1+C)s}$, its nearest point projection on $\xi$ is $e$ since $\xi$ lies entirely in the ball $B_{(1+C)s}$. According to Lemma~\ref{concate}, the concatenation $\xi \cup \bar{\beta}|_{\ge (1+C)s}$ is a $(3M, M)$--quasi-geodesic. 

Thus, we can redirect $\alpha$ to the geodesic ray $\bar{\beta}$ at radius $s$.
Since this can be done for an increasing sequence of radius $s$ (i.e, the sequence $s_1 < s_2 < \ldots $), it follows from Lemma~\ref{lem:infinitepoints} that $\alpha$ can be $(3M, M)$--quasi-redirected to $\bar{\beta}$. 

In conclusion, enlarging constants $M$ if necessary, we have that $\alpha$ can be $(3M, M)$--quasi-redirected to  $\bar{\beta}$ in both Case~1 and Case~2.

Since $\bar{\beta}$ can be $(3q, Q)$--quasi-redirected to $\beta$ by Lemma~\ref{lem:geodesicrepresentativeunder}, it follows from \cite[Lemma 3.2]{QR24} that $\alpha$ can be $(3q +    3M +1, Q+ M)$--quasi-redirected to $\beta$.

Using symmetric argument, it can be shown that $\beta$ is $(3q+ 3M +1, Q + M)$--quasi-redirected to $\alpha$.  Therefore $$\alpha \sim \beta$$ 
In particular, the poset $P(X)$ consists of exactly one point and satisfies all three QR-Assumptions. Therefore $\partial_{*} X$ is well-defined and  consists of only one point.

\end{proof}

\section{Quasi-redirecting boundary of finitely generated 3-manifold groups}
\label{sec:3-manifold}

In this section, we are going to prove Theorem~\ref{thm:main2} which says all finitely generated 3-manifold groups have well-defined QR-boundaries.

A compact orientable irreducible $3$--manifold $M$ with empty or tori boundary is called \emph{geometric} if its interior admits a geometric structure in the sense of Thurston which are $3$--sphere, Euclidean $3$--space, hyperbolic $3$--space, $S^{2} \times \mathbb R$, $\mathbb{H}^{2} \times \mathbb{R}$, $\widetilde{SL}(2, \mathbb{R})$, Nil and Sol. Otherwise, it is called \textit{non-geometric}. The QR-boundaries have been studied in \cite{NQ25}.

\begin{thm} \cite[Theorem A]{NQ25}
\label{thm:nongeometric}
Let $M$ be an non-geometric 3-manifold. Then the fundamental group $G = \pi_1(M)$ satisfies
the QR-Assumptions and hence  $\partial_{*} G$ is well-defined.    
\end{thm}

We now address the geometric case.

\begin{lem} [ Geometric case]
\label{lem:geometric}
Let $M$ be an geometric 3-manifold. Then  $\partial_{*} \pi_1(M)$ is well-defined.      
\end{lem}

\begin{proof}
 $M$ has a geometric structure modeled on eight geometries in the sense of Thurston: $S^3$, $\mathbb{R}^{3}$, $S^{2} \times \mathbb{R}$, Nil, $\widetilde{SL(2,\mathbb{R})}$, $\mathbb{H}^2\times \mathbb{R}$, $\mathbb{H}^{3}$, and Sol.    

 \begin{enumerate}
    \item If the geometry of $M$ is spherical, then its fundamental group is finite, and hence $\partial_{*} (\pi_1 (M))$ is empty.

    \item  If the geometry of $M$ is $\mathbb{E}^3$, then there exists a finite index subgroup $K \le \pi_1(M)$ such that $K$ is isomorphic to $\mathbb{Z}^3$. It follows that $\pi_1(M)$ is quasi-isometric to $\mathbb{E}^3$, and hence $\partial_{*} (\pi_1(M))$ consists of only one point as $\mathbb E^3$ has linear divergence.

    \item If the geometry of $M$ is $S^2 \times \mathbb{R}$, then there exists a finite index subgroup $K \le \pi_1(M)$ such that $K$ is isomorphic to $\mathbb Z$. It follows that $\pi_1(M)$ is quasi-isometric to $\mathbb{Z}$, and hence $\partial_{*} (\pi_1(M))$ is well-defined and consists of only two points.

    \item If the geometry of $M$ is $\mathbb{H}^2 \times \mathbb R$, then $M$ is finitely covered by $M' = \Sigma \times S^1$ where $\Sigma$ is a compact surface with negative Euler characteristic. It follows that $\pi_1(M)$ is quasi-isometric to the direct product $\pi_1(\Sigma) \times \mathbb Z$, and hence $\partial_{*} (\pi_1(M))$ is well-defined and consists of only one point.

    \item If $M$ has a geometry modeled on $\widetilde{SL(2,\mathbb R)}$, then $\partial_{*} (\pi_1(M))$ is well-defined and consists of only one point since two geometries $\mathbb{H}^2 \times \mathbb R$ and  $\widetilde{SL(2,\mathbb R)}$ are quasi-isometric and QR-boundary is a quasi-isometric invariant (Theorem~\ref{thm:collectfacts}).

    \item If $M$ has a geometry modeled on $\mathbb{H}^3$ then $M$ is a hyperbolic 3-manifold with finite volume. $\partial_{*} (\pi_1(M))$ is well-defined as shown in \cite{QR24}.

    \item Finally, If $M$ has a geometry modeled on Sol or Nil then $\pi_1(M)$ has linear divergence (see the first and second paragraphs in the proof of \cite[Theorem 4]{Ger94}, and thus Theorem~\ref{thm:main1} implies that $\partial_{*} (\pi_1(M))$ is well-defined and consists of only one point.
\end{enumerate}
\end{proof}

Below, we explain how one might reduce the study of all finitely generated 3–manifold
groups  to the case of compact, orientable, irreducible, $\partial$--irreducible 3–manifold groups.

Let $M$ be 3-manifold with finitely generated fundamental group. 
By Scott's Core
Theorem, $M$ contains a compact codimension zero submanifold whose inclusion map is a homotopy equivalence \cite{Sco73}, and thus also an isomorphism on fundamental groups. We thus can assume $M$ is compact. Since QR-boundary is a quasi-isometric invariant (Theorem~\ref{thm:collectfacts}), we can assume that M is orientable by passing to a double
cover if necessary. 

We can also assume that $M$ is irreducible and $\partial$--irreducible by the following reason:
Since $M$ is a compact, orientable $3$--manifold, it decomposes into irreducible, $\partial$--irreducible pieces $M_1, \dots, M_k$ (by the sphere-disc decomposition). In particular, $\pi_1(M)$ is the free product $$
\pi_1(M) = \pi_1(M_1)* \pi_1(M_2)* \cdots *\pi_1(M_k)
$$ Let $G_i = \pi_1(M_i)$. We remark here that $\pi_1(M)$ is hyperbolic relative to the collection $\mathbb{P} = \{G_1, \cdots, G_k \}$. According to \cite[Theorem D]{NQ25}, if the QR-boundary exist for each peripheral subgroup $G_i$ then the QR-boundary of $\pi_1(M)$ exists. In other
words, for the purpose of showing QR-boundary exists for $\pi_1(M)$, we only need to focus on the case where the manifold $M$ is compact, connected, orientable, irreducible, and $\partial$--irreducible.

If $M$ has empty or toroidal boundary, the existence of QR-boundary of $\pi_1(M)$ would follow from Theorem~\ref{thm:nongeometric} and Lemma~\ref{lem:geometric}.

We note that the compact, connected, orientable, irreducible and $\partial$--irreducible manifold $M$ could have boundary components that are higher genus surfaces. The following lemma addresses certain
manifolds with higher genus boundary.

\begin{lem}
    \label{lem:highergenus}
    Let $M$ be a compact, orientable, irreducible, $\partial$--irreducible $3$--manifold which has at least one boundary component of genus at least $2$.  Then the fundamental group $G = \pi_1(M)$ satisfies
the QR-Assumptions and hence  $\partial_{*} G$ is well-defined.    
\end{lem}

\begin{proof}
 As in 
  \cite[Section~6.3]{Sun20}, we can paste compact hyperbolic 3-manifolds with totally geodesic boundaries to the higher genus boundary components of $M$ to obtain a finite volume hyperbolic manifold $N$ (in case $M$ has trivial torus decomposition) or a mixed 3-manifold (in case $M$  has non-trivial torus decomposition). By \cite[Theorem A]{NQ25}, $\pi_1(N)$ has well-defined QR-boundary. The new manifold $N$ satisfies the following properties.
\begin{enumerate}
    \item \label{item1} $M$ is a submanifold of $N$ with incompressible tori boundary.
    \item  \label{item2} The torus decomposition of $M$ also gives the torus decomposition of $N$.
    \item \label{item3} Each piece of $M$ with a boundary component of genus at least $2$ is contained in a hyperbolic piece of $N$.
\end{enumerate}

Since $N$ contains at least one hyperbolic piece, we equip $N$ with a nonpositively curved metric as in \cite{Leeb95}. This metric induces a metric on the universal cover $\tilde{N}$. 

Let $M'_{1}, \dots, M'_{k}$ be the pieces of $M$ that satisfies (\ref{item3}). Let $N'_i$ be the hyperbolic piece of $N$ such that $M'_i$ is contained in $N'_i$. We remark here that it has been proved in \cite{NS19} that the inclusion of the subgroup $\pi_1(M)$ in $\pi_1(N)$ is a quasi-isometric embedding (see the proof of Case~1.2 in the proof of Theorem~1.3 in \cite{NS19}), and hence the inclusion $\widetilde{M} \to \widetilde{N}$ is a $(\lambda_1, c_1)$--quasi-isometric embedding for some uniform constants $\lambda_1, c_1$.

Since QR-boundary is a quasi-isometric invariant (see Theorem~\ref{thm:collectfacts}), it suffices to show that the QR-boundary exists for the universal cover $\widetilde{M}$.

Fix a base point $o \in \widetilde{M}$. Let $\alpha$ and $\beta$ be two $\qq$--rays in $\widetilde{M}$ based at $o$.  Since the inclusion $\widetilde{M} \to \widetilde{N}$ is a $(\lambda_1, c_1)$--quasi-isometric embedding, $\alpha$ and $\beta$ are $\qq'$--rays in $\widetilde{N}$ for some $q' = q' (q, Q, \lambda_1, c_1)$ and $Q' = Q'(q, Q, \lambda_1, c_1)$.

\textit{Claim~1:} There exists a function $\mathcal{M} \colon [1, \infty) \times [0, \infty) \to [1, \infty) \times [0, \infty)$ such that the following holds.
At every radius $r$, if $\alpha$ can be quasi-redirected to $\beta$ at radius $r$ via a $(A, B)$--quasi-geodesic $\gamma$ in $\widetilde{N}$ then $\alpha$ can be quasi-redirected to $\beta$ at radius $r$ via a $\mathcal{M}(A, B)$--quasi-geodesic $\gamma'$ which lies entirely in $\widetilde{M} \subset \widetilde{N}$.

\begin{proof}[Proof of the claim] Indeed, let $s$ be the landing time of $\gamma$ on $\beta$. By the construction of the manifold $N$, the restriction of $\gamma$ on $[0,s]$  can be decomposed into a concatenation 
$$
\gamma|_{[0, s]}  = \alpha_1 \cdot \beta_1 \cdot \alpha_2 \cdot \beta_2 \cdots \alpha_{\ell} \cdot \beta_{\ell} \cdot \alpha_{\ell+1}
$$ 
such that:
\begin{itemize}
    \item For each $j$, the subpath $\alpha_j$ is a subset of $\tilde{M}$, and $\beta_j$ intersects $\tilde{M}$ only at its endpoints. Here $\alpha_1$ and $\alpha_{\ell+1}$ might degenerate to points.

    \item Moreover, there are pieces $\tilde{M}'_j$ and $\tilde{N}'_j$ of $\tilde{M}$ and $\tilde{N}$ respectively such that $\tilde{M}'_j \subset \tilde{N}'_j$, $\beta_{j} \subset \tilde{N}'_j$, and the endpoints of $\beta_j$ lies in $\widetilde{\Sigma}_j \subset \tilde{M}'_j$ where $\widetilde{\Sigma}_j$ is the universal cover of some boundary component of genus at least 2 of $M_j$.
\end{itemize} 

Since each inclusion $\pi_1(M_j) \to \pi_1(N_j)$ is a quasi-isometric embedding, and $\pi_1(N_j)$ is a hyperbolic group, the subgroup $\pi_1(M_j)$ is a quasi-convex in $\pi_1(N_j)$ by Morse Lemma. 

Since there are finitely many pieces $M_j$'s, it follows that there exists a function $\mathcal{M}' \colon [1, \infty) \times [0 , \infty) \to  [1, \infty) \times [0 , \infty)$ such that for every $(\lambda, c)$--quasi-geodesic $\beta$ in $\widetilde{N}_j$ with endpoints in $\widetilde{M}_j$ will eventually lies in the $\mathcal{M}'(\lambda, c)$--distance from a geodesic in $\widetilde{M}_j$ connecting the two endpoints of the quasi-geodesic $\beta$.
Thus each $(A, B)$--quasi-geodesic $\beta_j$ in $\widetilde{N}_j$  lies within a $R$-distance from a geodesic in $\widetilde{M}_j$ connecting two endpoints of $\beta_j$, denoted $\beta'_j$ in $\widetilde{M}_j$. Define 
$$
\bar{\gamma} : = \alpha_1 \cdot \beta'_1 \cdot \alpha_2 \cdot \beta'_2 \cdots \alpha_{\ell} \cdot \beta'_{\ell} \cdot \alpha_{\ell+1}
$$ 
and 
$$
\gamma' : = \bar{\gamma} \cup \gamma|_{[s, \infty)}
$$
that lies entirely in $\widetilde{M}$.  Since each $\beta_j$ lies within a $R$--distance from $\beta'_j$ and $\alpha_j \subset \gamma$ and $\gamma$ is a quasi-geodesic ray, it follows that  $\gamma'$ is a $\mathcal{M}(A, B))$--quasi-geodesic ray.
\end{proof}

\textbf{Claim~2:} $\widetilde{M}$ satisfies all three QR-Assumptions, and hence it has well-defined QR-boundary.

For every quasi-geodesic ray $\alpha$ in $\widetilde{M}$, it is also a quasi-geodesic ray in $\widetilde{N}$ since the inclusion $\widetilde{M} \to \widetilde{N}$ is a quasi-isometric embedding. In \cite{NQ25}, the class $[\alpha] \in P(\widetilde{N})$ contains a geodesic representative $\underline{\alpha}$ and this geodesic lies in $\widetilde{M}$ by our construction of $N$ (i.e, the torus decomposition of $M$ also gives the torus decomposition of $N$) . We thus consider $\underline{\alpha}$ is a representative of $[a] \in P(\widetilde{M})$ as well.

Given $\mathbf{a} \in P(\widetilde{M})$, we consider $i(\mathbf{a}) \in P(\widetilde{N})$ where $i \colon P(\widetilde{M}) \to P(\widetilde{N})$ is the inclusion. By \cite[Theorem A]{NQ25}  there is a function 
\[
f_{i(\mathbf{a})} : \, [1, \infty) \times [0, \infty) \to [1, \infty) \times [0, \infty), 
\] 
called the redirecting function of
the class $i(\mathbf{a})$. Now let $\mathbf{b} \in P(\widetilde{M})$ such that $\mathbf{b} \prec \mathbf{a}$. In particular, $\mathbf{b} \prec i(\mathbf{a})$  such that if $\mathbf{b} \prec i(\mathbf{a})$ where $\mathbf{b} \in P(\widetilde{N})$ then any ${\qq}$--ray $\beta \in \mathbf{b}$ can be
$f_{i(\mathbf{a})}({\qq})$--quasi-redirected to $\underline{a}$ via $f_{i(\mathbf{a})}(\qq)$--rays in $\widetilde{N}$. 
According to Claim~1, such $f_{i(\mathbf{a})}(\qq)$--rays in $\widetilde{N}$ can be modified to be $\mathcal{M}(f_{i(\mathbf{a})}(\qq))$--rays in $\widetilde{M}$ that quasi-redirect $\beta$ to $\underline{a}$.
This shows that $\widetilde{M}$ would satisfy all three QR-Assumptions and thus $\pi_1(M)$ has well-defined QR-boundary.    
\end{proof}

\begin{proof}[Proof of Theorem~\ref{thm:main2}]
By applying Scott's Core Theorem and, if needed, passing to a double cover, combined with the quasi-isometric invariance of the QR-boundary, we may assume $M$ is compact and orientable. We then employ sphere-disc decomposition to decompose $M$ into  irreducible, $\partial$-irreducible manifolds $M_1, M_2, \ldots, M_k$. Consequently, the fundamental group decomposes as a free product: $\pi_1(M) = \pi_1(M_1) * \pi_1(M_2) * \cdots * \pi_1(M_k)$. By \cite[Theorem D]{NQ25}, the existence of a QR-boundary for each $\pi_1(M_i)$ implies its existence for $\pi_1(M)$. This follows from Theorem~\ref{thm:nongeometric}, Lemma~\ref{lem:geometric}, and Lemma~\ref{lem:highergenus}, completing the proof.
\end{proof}

%\bibliographystyle{alpha}
%\bibliography{refs.bib}
\bibliographystyle{alpha}

%\Addresses
\end{document}